\documentclass[preprint,11pt]{elsarticle}     

\usepackage{amstext,amsfonts,amsmath,amssymb,color}
\usepackage{amsthm}
\usepackage{graphicx}

\usepackage{enumitem}

\makeatletter
\def\@author#1{\g@addto@macro\elsauthors{\normalsize%
\def\baselinestretch{1}%
\upshape\authorsep#1\unskip\textsuperscript{%
\ifx\@fnmark\@empty\else\unskip\sep\@fnmark\let\sep=,\fi
\ifx\@corref\@empty\else\unskip\sep\@corref\let\sep=,\fi}%
\def\authorsep{\unskip,\space}%
\global\let\@fnmark\@empty
\global\let\@corref\@empty  
\global\let\sep\@empty}%
\@eadauthor={#1}
}
\makeatother

\DeclareMathOperator{\diag}{diag}
\DeclareMathOperator{\mul}{mult}

\newtheorem{theorem}{Theorem}[section]

\newtheorem{corollary}[theorem]{Corollary}
\newtheorem{lemma}[theorem]{Lemma}

\theoremstyle{definition}

\newtheorem{remark}[theorem]{Remark}

\journal{the journal.}

\begin{document}

\begin{frontmatter}

\title{\bf Determining some graph joins  by the signless Laplacian spectrum}

\author{Jiachang Ye}
\ead{yejiachang12@163.com}
\address{School of Mathematical Sciences, Xiamen University,  Xiamen, 361005, China}

\author{Jianguo Qian}
\ead{jgqian@xmu.edu.cn}

\address{School of Mathematical Sciences, Xiamen University,  Xiamen, 361005, China\\ School of Mathematics and Statistics, Qinghai Minzu University,  Xining, 810007, China}

\author{Zoran Stani\' c\,*\footnote{* Corresponding author.}}
\ead{zstanic@matf.bg.ac.rs}
\address{Faculty of Mathematics, University of Belgrade,
Studentski trg 16, 11 000 Belgrade, Serbia}

\begin{abstract} A graph  is determined by its signless Laplacian spectrum  if there is no other non-isomorphic graph sharing the same signless Laplacian spectrum. Let  $C_l$, $P_l$, $K_l$ and $K_{s,l-s}$  be  the cycle, the path, the complete graph and the complete bipartite graph with $l$ vertices, respectively. We prove that  $$G\cong K_1\vee (C_{l_1}\cup C_{l_2}\cup\cdots \cup C_{l_t}\cup sK_1),$$ with $s\ge 0, t\ge 1, n\geq 22$, is determined by the signless Laplacian spectrum if and only if either $s=0$ or $s\ge 1$ and $l_i\ne 3$ holds for all $1\leq i\leq t$, where $n$ is the order of $G$, and  $\cup$ and $\vee$ stand for the disjoint union and the join of two graphs, respectively. Moreover, for $s\ge 1$ and  $l_t=3$, $K_1\vee (K_{1,3}\cup C_{l_1}\cup C_{l_2}\cup\cdots \cup C_{l_{t-1}}\cup (s-1)K_1)$ is fixed as a graph sharing the signless Laplacian spectrum with  $G$. This contribution extends some recently published results.
\end{abstract}


\end{frontmatter}

\section{Introduction}

Throughout this paper, $G=(V,E)$ is an undirected, simple graph without loops or multiple edges. The number of vertices is called the order of $G$ and usually denoted by $n$.  For the graphs $G$ and $H$, we denote their \textit{disjoint union} by $G\cup H$. In particular, a disjoint union of $s$ copies of $G$ is denoted by $sG$. Similarly, $G\vee H$ denotes their \textit{join}, i.e., the graph obtained by inserting an edge between every vertex of $G$ and every vertex of~$H$.

The  {\em signless Laplacian matrix} of $G$ is
$Q(G)=D(G)+A(G)$, where $D(G)$ and $A(G)$ are the diagonal matrix of vertex degrees  and
the standard adjacency matrix of $G$, respectively. The \textit{signless Laplacian eigenvalues} of $G$ are the eigenvalues of $Q(G)$, and they form the \textit{signless Laplacian spectrum} of $G$; we abbreviate them to $Q$-eigenvalues and $Q$-spectrum. Two non-isomorphic graphs are \textit{$Q$-cospectral} if they share the same $Q$-spectrum. Conversely, a graph $G$ is said to be \textit{determined by the signless Laplacian spectrum} (for short, \textit{$G$ is \textit{DQS}}) if there is no other non-isomorphic graph with the same $Q$-spectrum.

Identifying graphs that are, or are not, determined by the spectrum of a prescribed graph matrix
is one of the oldest and the most extensively studied problems in the entire spectral theory. One of the pioneering results on this topic, motivated by applications in chemistry,  was reported by G\"{u}nthard and Primas in middle 1950s~\cite{GP}. Since then, this problem has been extensively studied  by many scholars. Many results can be found in \cite{towI,towII,CSS,3,4,Liu2,WangWei,Ye2024two,9} and references therein. The previous list is focused on references that emphasize the results concerning the signless Laplacian. A nice survey and a discussion on cospectrality can be found in the first three references.

We write $C_n$, $P_n$, $K_n$ and  $K_{s, n-s}$ to denote the cycle, the path, the complete graph and the complete bipartite graph of order $n$, respectively. The main result of this paper reads as follows.

\begin{theorem}\label{11t} The graph $$G\cong K_1\vee (C_{l_1}\cup C_{l_2}\cup\cdots \cup C_{l_t}\cup sK_1),$$ with $s\ge 0, t\ge 1$ and at least $22$ vertices is DQS if and only if either $s=0$ or $s\ge 1$ and $l_i\ne 3~ (1\leq i\leq t)$.
	
	Moreover, for $s\ge 1$ and  $l_t=3$, $G$ is $Q$-cospectral with $K_1\vee (K_{1,3}\cup C_{l_1}\cup C_{l_2}\cup\cdots \cup C_{l_{t-1}}\cup (s-1)K_1)$.
\end{theorem}

This result can be seen as an extension of one of the results established in~\cite{Ye2024two} concerning the {$Q$-cospectrality} of $K_1\vee (C_l\cup (n-l-1)K_1)$. It also generalizes the results of \cite{20,Liu21} concerning the $Q$-cospectrality of joins of an isolated vertex with sets of vertex disjoint cycles. Finally, Theorem~\ref{11t} is closely related to the results reported in \cite{MH3,Liu11,Liu1,sage,Wang11,Zhou11} concerning similar graph products.

The remainder of the paper is organized as follows. Section~\ref{sec:ir} contains additional terminology and notation, along with some known and some initial results. In Section~\ref{sec3}, we explicitly compute the signless Laplacian spectrum of certain graph joins. The proof of Theorem~\ref{11t} is given in~Section~\ref{sec4}.

\section{Preparatory}\label{sec:ir}

We write $d_{G}(v)$  and $N_{G}(v)$ to denote
the degree of a vertex $v$  and the set of neighbours of the same vertex in a graph $G$, respectively. For a vertex subset $U\subset V$,  $G[U]$  denotes the subgraph induced by $U$. The graphs obtained by deleting vertex $v$ and edge $e$ of $G$ are denoted by $G-v$ and $G-e$, respectively. Since $Q(G)$ is positive semidefinite, its eigenvalues can be arranged
as
$$0\le\kappa_{n}(G)\leq\kappa_{n-1}(G)\leq\cdots\leq\kappa_{1}(G).$$
It is well known that $\kappa_{n}(G)=0$ if and only if $G$ has a bipartite component~\cite[Theorem~1.18]{ifge}. Henceforth, we suppose that the vertex degrees of $G$ are arranged as  $d_1(G)\geq d_2(G)\geq \cdots\geq d_n(G)$; in particular, we suppose that the degree $d_i(G)$ is attained by a vertex $v_i$, for $1\leq i\leq n$.  If there is no risk of confusion, we will suppress $G$ in the previous notation. Besides, we  denote by $n_k(G)$  the number of vertices of degree $k$ (for short, $k$-vertices) in  $V(G)\setminus \{v_1\}$, that is $$n_k(G)=|\{u\,:\, d(u)=k~\text{and}~u\in V(G)\setminus \{v_1\}\}|,$$
where $v_1$ is a vertex attaining the maximum degree.

What's more, let  $\diag(c_{1},c_{2},\ldots,c_{n})$ stand for a diagonal matrix with $c_{1},c_{2},\ldots,c_{n}$ on the main diagonal, and let $\theta_{i}(N)$, $1\le i\le n$, denote the $i$th largest eigenvalue of an $n\times n$ real symmetric matrix~$N$.

We proceed with some known results.

\begin{lemma}\label{21l} {\rm\cite{Liu1}}  Let  $U=\{u_{1},u_{2},\ldots,u_s\}$ be a vertex subset of a graph $G$ and $H\cong G[U]$. If  $0\leq c_{j}\leq  d_G(u_{j}),$ $1\leq j\leq s$, then
	$\kappa_{i}(G)\geq \theta_{i} \big(\diag(c_{1},c_{2},\ldots,c_{s})+A(H)\big)$ holds for $1\leq i\leq s$.
\end{lemma}

\begin{lemma}\label{26l} {\rm\cite{Heu1}} If $G$ is a graph with $n~(n\geq 3)$ vertices and $e$ is an edge of $G$, then
	$\kappa_{1}(G)\ge \kappa_{1}(G-e)\ge \kappa_2(G)\ge \kappa_2(G-e)\cdots\ge \kappa_{n}(G)\ge \kappa_{n}(G-e)\ge0$.
\end{lemma}

\begin{lemma}\label{l2.10} {\rm\cite{Liu21}} If $G$ is a graph with $n\geq 12$ vertices such that  $0<\kappa_{n}\leq \kappa_{2}\leq 5<n<\kappa_{1}$, then $G$ is connected with $d_{2}\leq4$ and $d_{1}\geq n-3$.
\end{lemma}

\begin{lemma}\label{23l} {\rm\cite{Ye2024two}} Let $G$ be a connected  graph with $n$ vertices such that  $d_2\leq 4.$ If either $d_n=1<11\le d_1$ or $2\le d_n<8\le d_1$, then $d_1\ge \kappa_1-3$.
\end{lemma}

Let  $\triangledown(G)$ denote the number of triangles in graph $G$.

\begin{lemma}\label{24l} {\rm\cite{18}} If  $G$ contains $\triangledown(G)$ triangles, $n$ vertices and  $m$ edges,  then  $$ \sum_{i=1}^{n}\kappa^{3}_{i}=3\sum_{i=1}^{n}d^{2}_{i}+\sum_{i=1}^{n}d^{3}_{i}+6\triangledown(G), \hspace{5pt}\sum_{i=1}^{n}\kappa^{2}_{i}=2m+\sum_{i=1}^{n}d^{2}_{i},\hspace{5pt}\text{and}\hspace{5pt}
	\sum_{i=1}^{n}\kappa_{i}=\sum_{i=1}^{n}d_{i}=2m.$$
\end{lemma}

We next  deal with vertex-weighted graphs illustrated in Figure \ref{fig1}. They are `vertex-weighted' in the sense that some of their vertices are accompanied with weights. Precisely, vertices $w_1$ and $w_2$ have weight $3$, and weight of $v_1$ (labeled as $d(v_1)$ simply) is as in the figure. The graph $F_i$ can be seen as an induced subgraph of a larger graph, say $F'_i$, in which degrees of the previously mentioned vertices are equal to the corresponding weights. It is not necessary to know the exact distribution of other edges, but we care about vertex degrees, which is why we assign weights to them. To avoid unnecessary complicating, we use the same notation for vertex degrees and vertex weights. It follows that the signless Laplacian matrix representation $F_i$, say $Q^{*}(F_i)$, is the principal submatrix of the signless Laplacian $Q(F'_i)$.

\begin{remark} In the following statements, we consider a graph $G$ containing some of graphs $F'_i, 1\leq i\leq 4$, as subgraphs, and in each situation we estimate $\kappa_2(G)$ as demonstrated in this remark. Take, for example, that a graph $G$ contains $F'_2$ as a subgraph. Note that $w_1$ has weight $3$ in $F_2$, which implies the existence of a vertex, say $z$, such that $z\in V(F'_2)\setminus V(F_2)$ and $zw_1\in E(F'_2)$. We have the same situation for $w_2$. Since
	$$Q^*(F_2)=\begin{pmatrix} d_{F'_2}(v_1)& 1& 1& 1&1 \\ 1& 4& 1&1&1\\1& 1&3&0&0\\1&1&0&3&0\\1&1&0&0&2 \end{pmatrix}$$
	is a principal submatrix of $Q(F'_2)$, by virtue of Lemmas~\ref{21l} and~\ref{26l}, we have $\kappa_2(G)\ge  \theta_2(Q^*(F_2))>5$.
\end{remark}

\begin{figure}
	\centering
	\includegraphics[width=0.90\textwidth,angle=0]{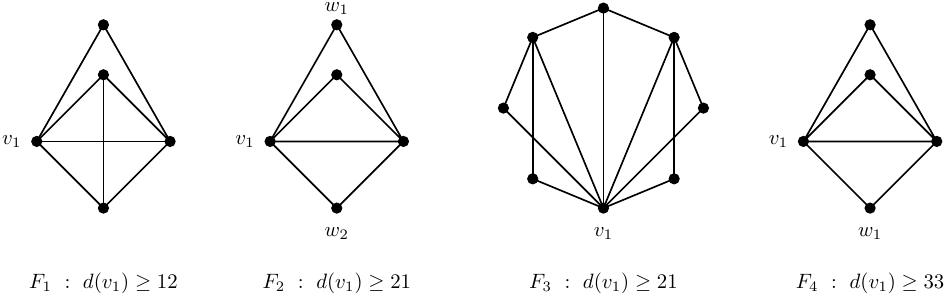}
	\caption{The   graphs    $F_{1}$, $F_2$, $F_3$ and $F_{4}$.}
	\label{fig1}
\end{figure}

We also quote the following lemma. (As designated above, $d_i(G)$ is attained by a vertex $v_i$.)


\begin{lemma}\label{2.6l}{\rm\cite{Ye2024two}} Let $G$ be a connected graph with $n~(n\ge 14)$ vertices such that $\kappa_2(G)\le 5$,  $d_1(G)\ge n-2$ and  $d_2(G)\le 4$. The following statements hold true:
	\begin{itemize}
		\item[(i)] If $w\not\in N_G(v_1)\cup \{v_1\}$, then $d_G(w)\le 3$ and  $w$ is adjacent to at most two $4$-vertices;
		\item[(ii)] Every two $4$-vertices (if any) are non-adjacent.
	\end{itemize}
\end{lemma}

We continue with a corollary.

\begin{corollary}\label{31cl} Let $G$  be a graph  with $n$ vertices such that $\kappa_2(G)\le 5$,   $d_1(G)=n-1$ and  $d_2(G)\le 4$.
	\begin{itemize}
		\item[(i)] If $n\ge 22$,  then every $4$-vertex is adjacent to at most one $3$-vertex;
		\item[(ii)] If $n\ge 22$, then any pair of $4$-vertices $u$ and $v$ (if any) satisfies $N_G(u)\cap N_G(v)=\{v_1\}$;
		\item[(iii)]  If $n\ge 34$, then each $4$-vertex is non-adjacent to any $3$-vertex.
	\end{itemize}
\end{corollary}
\begin{proof} (i): Suppose that $v$ is a $4$-vertex and $N_G(v)=\{v_1,w_1,w_2,w_3\}$.
	
	If $v$  is adjacent to exactly two 3-vertices, say $w_1$ and $w_2$, then we  consider two subcases: $w_1w_2 \in E(G)$ and $w_1w_2 \not \in E(G)$. In the former case, we have $\kappa_2(G)\ge \theta_2(Q^*(F_1))>5$ when $n\ge 22$, a contradiction. In the latter case, we have $\kappa_2(G)\ge \theta_2(Q^*(F_2))>5$ when $n\ge 13$, a contradiction.
	
	If $v$  is adjacent to exactly three 3-vertices $w_1,w_2,w_3$, due to Lemma~\ref{2.6l} we  need to consider two possibilities: $G[\{w_1,w_2,w_3\}]\cong P_2\cup P_1$ or $G[\{w_1,w_2,w_3\}]\cong3P_1$. By Lemmas~\ref{21l} and~\ref{26l}, we  have $\kappa_2(G)\ge \theta_2(Q^*(F_1))>5$ or $\kappa_2(G)\ge \theta_2(Q^*(F_2))>5$ when $n\ge 22$, which contradicts the assumptions of this lemma. Hence, (i) holds.
	
	(ii): Let $u$ and $v$ be two 4-vertices of $G$. By Lemma~\ref{2.6l}, $uv\not\in E(G)$. Note that $|N(u)\cap N(v)|\le 2$ when $n\ge 22$, since every $4$-vertex is adjacent to at most one 3-vertex  (by the same lemma) and  two $4$-vertices are non-adjacent (by Lemma~\ref{2.6l}).  If $|N(u)\cap N(v)|=2$, then $\kappa_2(G)\ge \theta_2(Q^*(F_{3}))>5$ when $n\ge 22$, which is impossible. Therefore,  $|N(u)\cap N(v)|=1$. Since $d_1=n-1$, the previous neighbourhoods necessarily intersect at $v_1$. This completes (ii).
	
	Finally, if $n\ge 34$ and  there exists a $4$-vertex $v$  adjacent to a 3-vertex, then  $\kappa_2(G)\ge \theta_2(Q^*(F_{4}))>5$, and~(iii) follows.
\end{proof}


\section{Signless Laplacian spectra of some joins}\label{sec3}

In this section, we compute the $Q$-spectrum of  $$K_1\vee (C_{l_1}\cup C_{l_2}\cup\cdots \cup C_{l_t}\cup sK_1)$$
and $$K_1\vee (K_{1,3}\cup C_{l_1}\cup C_{l_2}\cup\cdots \cup C_{l_{t-1}}\cup (s-1)K_1).$$

The $Q$-spectrum of $G$ is denoted by $S_Q(G)$; it is considered as a multiset, of course. Let $c^{(s)}$ and $\mul_{G}(c)$ denote the $s$ copies of a real number $c$ and the multiplicity of an eigenvalue~$c$ in $S_Q(G)$, respectively.

The first lemma tells us about the multiplicity of $1$ in the $Q$-spectrum of a graph containing vertices that share a particular neighbourhood. Although we believe it is known to the reader, a short proof is provided.

\begin{lemma}\label{l3.1} Let $G$ be a  graph of order $n~(n\ge 2)$ such that $N_G(z_1)=N_G(z_2)=\cdots =N_G(z_s)=\{z_{s+1}\}$  where $z_i\in V(G)$, for $s\ge 1$ and $1\leq i\leq s+1$. Then $\mul_{G}(1)\ge s-1$.
\end{lemma}
\begin{proof} Let $V(G)=\{z_1,z_2,\ldots, z_n\}$, for $n\ge s+1$. We fix ${\boldsymbol \psi_j}=(\psi_{j1},\psi_{j2},\ldots,\psi_{jn})^{\intercal}$, $1\leq j\leq s-1$, with $\psi_{jj}=1$, $\psi_{j,j+1}=-1$ and $\psi_{ji}=0$ for $i\notin \{j,j+1\}$. It is easy to check that $Q(G){\boldsymbol \psi_j}={\boldsymbol \psi_j}$. In addition, ${\boldsymbol \psi_1}, {\boldsymbol \psi_2}, \ldots, {\boldsymbol \psi_{s-1}}$ are linearly independent. Hence, $\mul_{G}(1)\ge s-1$.
\end{proof}

We consider the $Q$-spectrum of the former graphs.

\begin{lemma}\label{l3.2}Let $G\cong K_1\vee (C_{l_1}\cup C_{l_2}\cup\cdots \cup C_{l_t}\cup sK_1)$, $s,t\ge 1$. If $n$ is the order of $G$, then
	$$S_Q(G)=\left\{5^{(t-1)},1^{(s-1)},r_1,r_2,r_3, 3+2\cos\frac{2k\pi}{l_i}~:~ 1\le k\le l_i-1, 1\le i\le t\right\},$$
	where $r_1$, $r_2$ and $r_3$ are the  roots of $x^3-(n+5)x^2+5nx-4(n-s-1)$, and they satisfy $r_1>n\ge 5>r_2>2>1>r_3>0$.
\end{lemma}

\begin{proof} As in Lemma~\ref{l3.1}, we suppose that $V(G)=\{z_i\,:\, 1\leq i\leq n\}$, where $d_G(z_n)=n-1$ and $d_G(z_i)=1$ for $1\leq i\leq s$. By the same lemma, $\mul_{G}(1)\ge s-1$. 	In what follows, we construct the eigenvectors for the remaining eigenvalues.

We first deal with the eigenvalue 5. Let ${\boldsymbol \zeta_j}=(\zeta_{j}(z_1),\zeta_{j}(z_2),\ldots,\zeta_{j}(z_n))^\intercal$ for $1\leq j\leq t-1$, where $$\zeta_{j}(v)=\left\{\begin{array}{rl}-l_{j+1}/l_1,&\text{for}~v\in V(C_{l_1}),\\ 1,& \text{for}~v\in V(C_{l_{j+1}}),\\ 0,&\text{otherwise}.\end{array}\right.$$

It is easy to check that this comprises a set of linearly independant eigenvectors associated with the eigenvalue $5$ of $Q(G)$. Thus, $\mul_{G}(5)\ge t-1$.

Let now  $V(C_{l_i})=\{w_{i1},w_{i2},\ldots,w_{il_i}\}$, $1\leq i\leq t$, and $A(C_{l_{i}}){\boldsymbol \varphi}=\rho {\boldsymbol \varphi}$ (${\boldsymbol \varphi}\ne {\bf 0}$), i.e., ${\boldsymbol \varphi}$ is an eigenvector for $A(C_{l_{i}})$ corresponding  to the eigenvalue $\rho$. It is well known that $\rho=2\cos(2k\pi/l_i)$, $0\leq k\leq l_i-1$. Moreover, by excluding $k=0$, we get that ${\boldsymbol \phi}=(\phi(z_{1}), \phi(z_{2}),\ldots,\phi(z_{n}))^\intercal$ with
$$\phi(v)=\left\{\begin{array}{rl}\varphi(v),&\text{for}~v\in V(C_{l_{i}}),\\ 0,& \text{otherwise},\end{array}\right.$$
is an eigenvector for $Q(G)$ corresponding to the eigenvalue $3+\rho$. To see this, one may observe that $I_{l_{i}}+Q(C_{l_{i}})=3I_{l_{i}}+A(C_{l_{i}})$ is a principal submatrix of $Q(G)$. Therefore, $3+2\cos(2k\pi/l_i)$, $1\leq k\leq l_{i}-1$,  are the eigenvalues of $Q(G)$, for every $i~(1\leq i\leq t)$.

	We now show that $r_1$, $r_2$ and $r_3$ are the eigenvalues of $Q(G)$.  Suppose $Q(G){\boldsymbol x}=c{\boldsymbol x}$ (${\boldsymbol x}\ne {\bf 0}$), where ${\boldsymbol x}=(x(z_1),x(z_2),\ldots,x(z_n))^\intercal$. Then for $z_n$ (the maximum degree vertex), we have
	\begin{align}\label{e4.1}  (c-n+1)x(z_n)=\sum_{i=1}^{n-1}(x(z_i)). \end{align}
	For the $s$ pendant vertices $z_1,z_2,\ldots,z_s$, we have
	\begin{align}\label{e4.2}
		(c-1)x(z_i)=x(z_n), 1\leq i\leq s.
	\end{align}
	So, if $c\ne 1$, then \begin{align}\label{e4.3}
		x(z_i)=\frac{x(z_n)}{c-1}, 1\leq i\leq s.
	\end{align}
	For the vertices in $V(C_{l_i})=\{w_{i1},w_{i2},\ldots,w_{il_i}\}$, $1\leq i\leq t$, we have
	\begin{equation*}\label{e4.4}\left\{
		\begin{array}{ll} (c-3)x(w_{i1})=x(w_{i2})+x(w_{i,l_i})+x(z_n),\\
			(c-3)x(w_{i2})=x(w_{i3})+x(w_{i,1})+x(z_n),\\
			\quad\quad\quad\quad \vdots \\
			(c-3)x(w_{i,l_i})=x(w_{i1})+x(w_{i,l_i-1})+x(z_n).
		\end{array}
		\right.\end{equation*}
	We set $x(w_{i1})=x(w_{i2})=\cdots =x(w_{i,l_i})$. If $c\ne 5$, then we have
	\begin{align}\label{e4.5}
		x(w_{ij})=\frac{x(z_n)}{c-5}, 1\leq j\leq l_i, 1\leq i\leq t.
	\end{align}
	Combining \eqref{e4.1}, \eqref{e4.3} and \eqref{e4.5}, we find
	$$(c-n+1)x(z_n)=s\frac{x(z_n)}{c-1}+(n-s-1)\frac{x(z_n)}{c-5}.$$
	For $x(z_n)\ne 0$, we compute $f(c)=c^3-(n+5)c^2+5nc-4(n-s-1)=0$. It is verified directly that the corresponding roots, $r_1,r_2$ and $r_3$, satisfy $r_1>n\ge 5>r_2>2>1>r_3>0$. In addition, $${\boldsymbol \eta_i}=\Big(\big(\frac{1}{r_i-1}\big)^{(s)},\big(\frac{1}{r_i-5}\big)^{(n-s-1)},1\Big)^\intercal$$
	are the corresponding eigenvectors.
	
	Gathering the previous results, we find that $r_1,r_2,r_3$ and  $3+2\cos\frac{2k\pi}{l_i}$, for $1\le k\le l_i-1, 1\le i\le t$, feature as the $Q$-eigenvalues of $G$.  In this context, one may observe that if $r_j=3+2\cos\frac{2k\pi}{l_i}$ occurs for some choice of $i, j ,k$, then the corresponding eigenvectors are linearly independant. This comprises the multiset of $n-(s-1)-(t-1)$ $Q$-eigenvalues. Moreover, we have obtained that 1 and 5 are the $Q$-eigenvalues with multiplicity at least $s-1$ and $t-1$, respectively. Again, if $1=3+2\cos\frac{2k\pi}{l_i}$, then the eigenvectors are linearly independant (those for $1$ are constructed in Lemma~\ref{l3.1}). Obviously, the previous lower bounds are attained, and we are done.
\end{proof}

\begin{remark}\label{r3.3} Note that  $3+2\cos(2k\pi/l_i)=1$ holds if and only if $l_i$ is even and $k=l_i/2$. Thus, $\mul_G(1)=s-1+c_{e}(C)$, where $c_{e}(C)$ denotes the number of even cycles among $C_{l_i}$. Besides, $n<r_1=\kappa_1(G)<n+1$ when $n\ge 7$ and $0<r_3=\kappa_n(G)<1$. Last but not least, for any given order $n$, $\kappa_1(G)$ and $\kappa_n(G)$ depend on the number $s$ of pendant vertices, but do not depend on the number of cycles nor the length of each cycle.
\end{remark}

We proceed with the latter graph.

\begin{lemma}\label{l3.4} Let $H\cong K_1\vee (K_{1,3}\cup C_{l_1}\cup C_{l_2}\cup\cdots \cup C_{l_{t-1}}\cup (s-1)K_1)$, $s\geq 1, t\geq 2$. If the order of $H$ is $n$, then
	$$S_Q(H)=\left\{5^{(t-1)},1^{(s-1)},2^{(2)},r_1,r_2,r_3, 3+2\cos\frac{2k\pi}{l_i}~:~ 1\le k\le l_i-1, 1\le i\le t-1\right\},$$
	where $r_1$, $r_2$ and $r_3$ are the roots of $x^3-(n+5)x^2+5nx-4(n-s-1)$, and they satisfy $r_1>n> 5>r_2>2>1>r_3>0$.
	
	In particular, for $H^*\cong K_1\vee (K_{1,3}\cup (s-1)K_1)$, $s\ge 1$, we have
	$$S_Q(H^*)=\left\{1^{(s-1)},2^{(2)},r_1,r_2,r_3\right\}.$$
\end{lemma}

\begin{proof} We only consider the case when $s,t\ge3$, as the remaining possibility is treated by following the same lines. For convenience, we set $V(H)=\{z_i\,:\, 1\leq i\leq n\}$, with $d_H(z_n)=n-1$, $d_H(z_{n-4})=4,d_H(z_{n-3})=d_H(z_{n-2})=d_H(z_{n-1})=2$ and $d_H(z_i)=1$, $1\leq i\leq s-1$. Note that $K_{1,3}$ is a component of $H-z_n$ with $V(K_{1,3})=\{z_{n-4},z_{n-3},z_{n-2},z_{n-1}\}$.
	
The eigenvectors ${\boldsymbol \psi_j}$, $1\leq j\leq s-2$, for $1$ are constructed as in the proof of Lemma~\ref{l3.1}, whereas the eigenvectors ${\boldsymbol \zeta_j}$, $1\leq j\leq t-2$, for $5$ and all the eigenvectors for $3+2\cos\frac{2k\pi}{l_i}, 1\le k\le l_i-1, 1\le i\le t-1$, are constructed as in the proof of Lemma~\ref{l3.2}. 	

Next,
$${\boldsymbol \psi_{s-1}}
=\big(2,0^{(n-6)},1,-1,-1,-1,0\big)^\intercal,$$
$${\boldsymbol \zeta_{t-1}}
=\big(0^{(s-1)},\big(-{6}/{l_1}\big)^{(l_1)},0^{(l_2)},\ldots,0^{(l_{t-1})},
3,1,1,1,0\big)^\intercal,$$
and
$$
\big(0^{(n-5)},0,1,-1,0,0\big)^\intercal,~ \big(0^{(n-5)},0,0,1,-1,0\big)^\intercal,$$
are associated with $1$, $5$ and $2^{(2)}$, respectively.

In addition, the roots, $r_1, r_2$ and $r_3$, of $x^3-(n+5)x^2+5nx-4(n-s-1)$ satisfy the inequalities given in the formulation of this lemma, and they appear as the $Q$-eigenvalues of $H$ afforded by 	$${\boldsymbol \eta_i}=\Big(\big(\frac{1}{r_i-1}\big)^{(s-1)},\big(\frac{1}{r_i-5}\big)^{(n-s-4)},
	\frac{r_i+1}{(r_i-1)(r_i-5)}, \big(\frac{r_i-3}{(r_i-1)(r_i-5)}\big)^{(3)},   1\Big)^\intercal,$$
	which is a result of a simple algebraic calculus.
	
Linear independence in matching cases are considered as in the proof of Lemma~\ref{l3.2}. Finally, the result for $H^*$ is extracted by setting $t=1$.
\end{proof}

From Lemmas~\ref{l3.2} and~\ref{l3.4}, we  immediately deduce the following corollary.
\begin{corollary}\label{c3.5} The graphs $$K_1\vee (C_3\cup C_{l_1}\cup C_{l_2}\cup\cdots \cup C_{l_{t-1}}\cup sK_1)~~\text{and}~~K_1\vee (K_{1,3}\cup C_{l_1}\cup C_{l_2}\cup\cdots \cup C_{l_{t-1}}\cup (s-1)K_1),$$ with  $s,t\ge 1$, are $Q$-cospectral.
\end{corollary}

We remark that the previous corollary covers a particular case $t=1$.

\section{Proof  of Theorem~\ref{11t}}\label{sec4}

Let $G$ be as in the formulation of Theorem~\ref{11t}. Two particular cases are resolved in the previous references: for $s=0$, $G$ is DQS by~\cite{Liu21}, while for $t=1$, the statement of the theorem follows from~\cite{Ye2024two}. Therefore, we suppose $s\ge 1$ and $t\ge 2$. Let further $H$ be a graph (if any) that is $Q$-cospectral with~$G$.

On the basis of Lemma~\ref{l3.2}, we deduce the following setting: $$0<\kappa_n(H)<1<\kappa_2(H)=5<n<\kappa_1(H).$$
We abbreviate $n_i(G)$, $d_j(G)$, $n_i(H)$ and $d_j(H)$ to $n_i$, $d_j$, $n_{i}^*$ and $d_{j}^*$, respectively; clearly, this applies for $1\leq i\leq n-1$ and $1\leq j\leq  n$.
From Lemma~\ref{l2.10}, we have that $H$ is connected with $d_2^*\le 4$ (this equality will be frequently used).

We  prove the following result.

\begin{lemma}\label{l4.1} If $n\ge 14$, then $d_1^*=d_1=n-1$.
\end{lemma}
\begin{proof} By Lemma~\ref{l2.10}, we have $d_1^*\ge n-3\ge 11$. By Lemma \ref{23l}, $n<\kappa_1(H)\le d_1^*+3$, and thus $d_1^*\ge n-2$. Suppose that $d_1^*=n-2$.
	
	From Lemma \ref{24l}, we have 	
	\begin{equation}\label{e51}\left\{
		\begin{array}{ll}  2n_3+6n_4+d_1(d_1-3)=2n_{3}^*+6n_{4}^*+d_{1}^*(d_{1}^*-3),\\ n_2-3n_4-d_1(d_1-4)=n_{2}^*-3n_{4}^*-d_{1}^*(d_{1}^*-4),\\
			2n_1+2n_4+d_1(d_1-5)=2n_{1}^*+2n_{4}^*+d_{1}^*(d_{1}^*-5).
		\end{array}
		\right.\end{equation}
	Since $d_1=n-1$, $d_1^*=n-2$, $n_2=0$ and $n_4=0$, from \eqref{e51}, we have
	\begin{equation}\label{e52}\left\{
		\begin{array}{ll}  n_{3}^*=2n-s-4-3n_{4}^*,\\
			n_{2}^*=3n_{4}^*-2n+7,\\
			n_{1}^*=n+s-4-n_{4}^*.
		\end{array}
		\right.\end{equation}
	Now, from \eqref{e52}, we find
	\begin{align}\label{e53} n_2^*+n_3^*=3-s\ge 0.
	\end{align}
	Therefore, $s\in\{1,2,3\}$. These cases are considered separately.
	
	\smallskip\noindent{\textit{Case 1: $s=1$.}}  Remark~\ref{r3.3} implies $\mul_G(1)\le (n-2)/4$, while  by~\eqref{e53}, we have $n_2^*+n_3^*=2$.
	
	\smallskip\noindent{\textit{Subcase 1.1: $n_2^*=0$ and $n_3^*=2$.}} By \eqref{e52}, we get $n_4^*=(2n-7)/3$ and $n_1^*=(n-2)/3$. Note that $d_1^*=n-2$. If $n>26$,  Lemma~\ref{l3.1} leads to the impossible scenario: $\mul_H(1)\ge n_1^*-2>(n-2)/4\ge \mul_G(1)$. If $14\le n\le 26$, we have $\kappa_2(H)>5$ since $n_4^*=(2n-7)/3$ and $d_1^*=n-2$ (by Lemma~\ref{2.6l}(ii)), but this contradicts $\kappa_2(H)=5$.
	
	\smallskip\noindent{\textit{Subcase 1.2: $n_2^*=1$ and $n_3^*=1$.}} Here, \eqref{e52} leads to $n_4^*=(2n-6)/3$ and $n_1^*=(n-3)/3$, and Lemma~\ref{l3.1} leads to $\mul_H(1)\ge n_1^*-2>(n-2)/4\ge \mul_G(1)$, whenever $n>30$. For $14\le n\le30$, as in the previous subcase we arrive at $\kappa_2(H)>5$ (since $n_4^*=(2n-6)/3$ and $d_1^*=n-2$ by Lemma \ref{2.6l}(ii)), which is impossible.
	
	\smallskip\noindent{\textit{Subcase 1.3: $n_2^*=2$ and $n_3^*=0$.}} From $n_4^*=(2n-5)/3$ and $n_1^*=(n-4)/3$, we arrive at $\mul_H(1)\ge n_1^*-2>(n-2)/4\ge \mul_G(1)$ for  $n>34$. The case $14\le n\le34$ is resolved as before.
	
	\smallskip\noindent{\textit{Case 2: $s=2$.}} Remark~\ref{r3.3} implies $\mul_G(1)\le (n+1)/4$, while~\eqref{e53} leads to $n_2^*+n_3^*=1$. There are two subcases ($(n_2^*, n_3^*)=(0,1)$ and $(n_2^*, n_3^*)=(1,0)$). Each is resolved as in the previous part of the proof: On the basis of~\eqref{e52} and Lemma \ref{l3.1}, we arrive at a contradiction treating $n\geq 23$ (for the former one) and $n\geq 27$ (for the latter one), whereas the remaining possibilities for $n$ are eliminated by Lemma~\ref{2.6l}(ii) and the condition $\kappa_2(H)=5$.
%
%
	
	\smallskip\noindent{\textit{Case 3: $s=3$.}}  Remark \ref{r3.3} gives $\mul_G(1)\le (n+4)/4$, and~\eqref{e53} gives  $n_2^*=n_3^*=0$.  Lemma~\ref{l3.1} eliminates the possibility $n>20$, and Lemma~\ref{2.6l}(ii) treats the remaining orders.
\end{proof}

Henceforth, $Z_n$ denotes a tree obtained by duplicating a pendant vertex of the path~$P_{n-1}$, $n\geq 4$.

\begin{lemma}\label{l4.2}  If $n\ge 22$, then $H$ does not contain $K_1\vee (Z_6\cup 15K_1)$ as a subgraph.
\end{lemma}

\begin{proof}
	Otherwise Lemmas~\ref{21l} and~\ref{26l} would imply $\kappa_2(H)\ge \kappa_2(K_1\vee (Z_6\cup   15K_1))>5$, as $n\ge 22$.
\end{proof}

Now, we prove the theorem.

\medskip\noindent{\em Proof of  Theorem \ref{11t}.} According to the discussion in the beginning of this section, we suppose that $s\geq 1, t\geq 2$. We also set $n\geq 22$ (as in the statement of this theorem) and $l_i\neq 3$. We shall prove that the graph $H$, introduced at the beginning of this section, is isomorphic to $G$.

 Lemma~\ref{l4.1} ensures $d_1^*=n-1$, and we denote the corresponding vertex by $v_1$. By inserting $n_3=n-s-1$, $n_1=s$, $n_2=n_4=0$ and $d_1=d_1^*$ in~\eqref{e51}, we arrive at \begin{align}\label{32e} n_3^*=n-s-1-3n_4^*,\,\, n_2^*=3n_4^*,\, \,\, n_1^*=s-n_4^*\,~\text{and}\,~\triangledown(G)=\triangledown(H)+n_4^*. \end{align}

 We claim that $G$ and $H$ share the same vertex degrees. To achieve this, we need  to verify that $n_4^*=0$, see \eqref{32e}. For a contradiction, suppose that $n_4^*=r\ge 1$, and let $z_1,z_2,\ldots,z_r$ be the $4$-vertices in $H$. From~\eqref{32e}, we deduce \begin{align}\label{33e} n_3^*=n-s-1-3r,~ n_2^*=3r,~ n_1^*=s-r~\,\text{and}\,~ \triangledown(G)=\triangledown(H)+r.
\end{align}
By Lemma \ref{2.6l} and Corollary~\ref{31cl},  every $z_i$ is adjacent to at most one 3-vertex, there is no edge between $z_i$ and $z_j$ ($i\neq j$), and  $N_H(z_i)\cap N_H(z_j)=\{v_1\}$. Together with Lemma~\ref{l4.2}, these structural observations lead to the conclusion that every component of  $H-v_1$ is isomorphic to either $Z_q~(4\le q\le 5)$, or $C_l~(l\ge 3)$, or $P_k~(k\ge 1)$. Using this, we obtain the following conclusions:

\begin{itemize}
	\item[(a)] From $n_4^*=r$, we deduce that there are  $r$ copies of $Z_q$ in $H-v_1$, that is $r_1$ copies of $Z_4$ and $r-r_1$ copies of $Z_5$. They give rise to $4r-r_1$ triangles in $H$.
	
	\item[(b)] Since $n_2^*=3r$, and from (a) we know that there are  $r$ copies of $Z_q~(4\le q\le 5)$ in $H-v_1$, which also means that there are already $3r$ $2$-vertices in these $Z_q$ in $H$. So there must be $k=1$ for all $P_k$ in $H-v_1$. Besides, $n_1^*=s-r$ implies the existence of $s-r$ copies of $P_1$ in $H-v_1$. Of course, they do not contribute any triangle in $H$.
	
	\item[(c)] The remaining components of $H-v_1$ are cycles which produce at least $n-s-1-4r+r_1$ triangles in $H$.
\end{itemize}

From (a), (b) and (c), we deduce  $$\triangledown(H)\ge (4r-r_1)+ (n-s-1-4r+r_1)=n-s-1.$$  On the other hand, we have $\triangledown(G)=n-s-1$. Therefore, $\triangledown(G)\not=\triangledown(H)+r$, which contravenes the last equality of~\eqref{33e}. This proves the claim, and then the only structural possibility is $$H\cong K_1\vee (C_{l'_1}\cup C_{l'_2}\cup\cdots \cup C_{l'_q}\cup sK_1).$$
First, $q=t$ follows because $\mul_H(5)=q-1=t-1=\mul_G(5)$. Secondly, by setting $l_{1}\ge l_{2} \ge\cdots \ge l_t>3$, $l'_{1}\ge l'_{2} \ge\cdots \ge l'_t\ge 3$ and observing that $\max\{3+2\cos\frac{2k\pi}{l_1}\,:\:1\leq k\leq l_1-1\}=3+2\cos\frac{2\pi}{l_1}$, we deduce $l_1=l'_1$. By excluding the $Q$-eigenvalues $3+2\cos\frac{2k\pi}{l_1},$ $1\leq k\leq l_1-1$, from the common $Q$-spectrum given in Lemma~\ref{l3.2}, we find $l_2=l'_2$. By repeating this procedure, we arrive at $l_i=l'_i$, for all $i$, which means that $H$ is isomorphic to $G$.

Finally, by allowing $l_t=3$, we find that $K_1\vee (C_3\cup C_{l_1}\cup C_{l_2}\cup\cdots \cup C_{l_{t-1}}\cup sK_1)$ and   $K_1\vee (K_{1,3}\cup C_{l_1}\cup C_{l_2}\cup\cdots \cup C_{l_{t-1}}\cup (s-1)K_1)$ are $Q$-cospectral, by Corollary \ref{c3.5}. The proof is completed.
\qed

%

\section*{Acknowledgements}
The
research was supported by the National Natural Science Foundation of China [No.~12361070] and the Science Fund of the Republic of Serbia,
grant number 7749676: Spectrally Constrained Signed Graphs with Applications in Coding
Theory and Control Theory -- SCSG-ctct.

\end{document}